\documentclass[oneside,english]{amsart}
\usepackage[T1]{fontenc}
\usepackage[latin9]{inputenc}
\usepackage{amsthm}
\usepackage{amstext}
\usepackage{amssymb}
\usepackage{esint}
\usepackage{graphicx}
\usepackage{enumerate}

\usepackage{graphicx}
\graphicspath{{noiseimages/}}

\makeatletter
\newtheorem{theorem}{Theorem}[section]
\newtheorem{lemma}[theorem]{Lemma}

\numberwithin{figure}{section}
\theoremstyle{definition}
\newtheorem{definition}[theorem]{Definition}

\theoremstyle{remark}
\newtheorem{remark}[theorem]{Remark}

\numberwithin{equation}{section}
\makeatother

\usepackage{babel}

	\DeclareMathOperator{\loc}{loc}
	
	\DeclareMathOperator*{\esssup}{ess\,sup}
	\DeclareMathOperator{\cp}{cap}
	
\begin{document}

\title[Spectral Properties]{Spectral Properties of the Neumann-Laplace operator in Quasiconformal Regular Domains}

\author{V.~Gol'dshtein, V.~Pchelintsev, A.~Ukhlov}

\begin{abstract}
In this paper we study spectral properties of the Neumann-Laplace operator in planar quasiconformal regular domains 
$\Omega\subset\mathbb R^2$. This study is based on the quasiconformal theory of composition operators on Sobolev spaces. Using the composition operators theory we obtain estimates of 
constants in Poincar\'e-Sobolev inequalities and as a consequence lower estimates of the first non-trivial eigenvalue of the Neumann-Laplace operator in planar quasiconformal regular domains.
\end{abstract}
\maketitle

\section{Introduction}

We study the spectral problem for the Laplace operator with the Neumann boundary condition in
planar quasiconformal regular domains  $\Omega \subset \mathbb{R}^2$.
The weak statement of this spectral problem is as follows: a function $u$ solves this problem iff 
$u \in W^{1}_{2}(\Omega)$ and
$$
\int\limits _{\Omega} \nabla u(x) \cdot \nabla v(x)\,dx =
\mu \int\limits _{\Omega} u(x) v(x)\,dx
$$ 
for all $v \in W^{1}_{2}(\Omega)$. 

We prove discreetness of the spectrum of the Neumann--Laplace operator in 
quasiconformal $\beta$-regular domains and obtain the lower estimates of the first non-trivial eigenvalue in the terms of quasiconformal geometry of domains:

\vskip 0.3cm
\noindent
{\bf Theorem A.}
\textit{Let $\Omega \subset \mathbb R^2$ be a $K$-quasiconformal $\beta$-regular domain. Then the spectrum of the Neumann--Laplace operator in $\Omega$ is discrete,
and can be written in the form of a non-decreasing sequence:
\[
0=\mu_0(\Omega)<\mu_1(\Omega)\leq \mu_2(\Omega)\leq \ldots \leq \mu_n(\Omega)\leq \ldots ,
\]
and
\[
\frac{1}{\mu_1(\Omega)} \leq 
\frac{4K}{\sqrt[\beta]{\pi}} \left(\frac{2\beta -1}{\beta -1}\right)^{\frac{2 \beta-1}{\beta}} \big\|J_{\varphi}\mid L_{\beta}(\mathbb D)\big\|,
\]
where $\varphi:\mathbb D \to \Omega$ is the $K$-quasiconformal mapping. }
\vskip 0.3cm

\begin{definition}
Let $\Omega \subset \mathbb{R}^2$ be a simply connected planar domain. Then $\Omega$ is called a $K$-quasiconformal $\beta$-regular domain if there exists a $K$-quasiconformal mapping $\varphi : \mathbb D \to \Omega$ such that
$$
\int\limits_\mathbb D |J(x, \varphi)|^{\beta}~dx < \infty \quad\text{for some}\quad \beta >1.
$$
The domain $\Omega \subset \mathbb{R}^2$ is called a $K$-quasiconformal regular domain if it is a $K$-quasiconformal $\beta$-regular domain for some $\beta>1$.
\end{definition}

The notion of quasiconformal regular domains is a generalization of the notion of conformal regular domains was introduced in~\cite{BGU15} and was used for study conformal spectral stability of the Laplace operator (see, also \cite{BGU16}). 

Recall that a homeomorphism $\varphi:\Omega \rightarrow \widetilde{\Omega}$
between planar domains $\Omega, \widetilde{\Omega}\subset\mathbb R^2$ is called a $K$-quasiconformal mapping if it preserves
orientation, belongs to the Sobolev class $W_{2,\loc}^{1}(\Omega)$
and its directional derivatives $D_{v}$ satisfy the distortion inequality
$$
\max\limits_{v: |v|=1}|D_{v}\varphi|\leq K\min_{{v: |v|=1}}|D_{v}\varphi|\,\,\,
\text{a.~e. in}\,\,\, \Omega \,.
$$

Note, that class of quasiconformal regular domain includes the class of Gehring domains \cite{AK} and can be described in terms of quasihyperbolic geometry \cite{KOT}. 

\begin{remark}
The notion of quasiconformal $\beta$-regular domains is more general then the notion of conformal $\alpha$-regular domains. Consider, for example, the unit square $\mathbb Q\subset\mathbb R^2$. Then $\mathbb Q$ is a conformal $\alpha$-regular domain for $2<\alpha\leq 4$ $(\alpha=2\beta)$ \cite{GU16} and is a quasiconformal $\beta$-regular domain for all $1<\beta\leq \infty$ because the unit square $\mathbb Q$ is quasiisometrically equivalent to the unit disc $\mathbb D$.
\end{remark}

\begin{remark}
Because $\varphi: \mathbb D \to \Omega$ is a quasiconformal mapping, then integrability of the derivative is equivalent to integrability of the Jacobian:
$$
\int\limits_\mathbb D |J(x,\varphi)|^{\beta}~dx\leq \int\limits_\mathbb D |D \varphi(x)|^{2\beta}~dx 
\leq K^{\beta}\int\limits_\mathbb D |J(x,\varphi)|^{\beta}~dx.
$$
\end{remark}

\vskip 0.3cm

In 1961 G.Polya  \cite{P60} obtained upper estimates for eigenvalues of Neumann-Laplace operator in so-called plane-covering domains. Namely, for the first eigenvalue:
$$
{\mu_1(\Omega)}\leq 4\pi|\Omega|^{-1}.
$$

The lower estimates for $\mu_1(\Omega)$ were known before only
for convex domains. In the classical work \cite{PW} it was proved that if $\Omega$ is convex with diameter $d(\Omega)$ (see, also \cite{ENT,FNT,V12}), then
\begin{equation*}
\label{eq:PW}
\mu_1(\Omega)\geq \frac{\pi^2}{d(\Omega)^2}.
\end{equation*}

In \cite{GU16} we proved, that if $\Omega\subset\mathbb R^2$ be a conformal regular domain, then the spectrum of Neumann-Laplace operator in $\Omega$ is discrete and the first non-trivial eigenvalue depends on hyperbolic geometry of the domain. Because quasiconformal mappings represent a more flexible class of mapping in the present paper we suggest an approach to the Poincar\'e-Sobolev inequalities which is based on the quasiconformal mappings theory in connection with the composition operators theory on Sobolev spaces.

\vskip 0.3cm

Theorem A is based on the Poincar\'e--Sobolev inequalities in quasiconformal regular domains:

\vskip 0.3cm
\noindent
{\bf Theorem B.}
\textit{Let $\Omega\subset\mathbb R^2$ be a $K$-quasiconformal $\beta$-regular domain. Then:
\begin{enumerate}[(1)]
\item the embedding operator
\[
i_{\Omega}:W_2^1(\Omega) \hookrightarrow L_s(\Omega)
\]
is compact for any $s \geq 1$;
\item for any function $f \in W^{1}_{2}(\Omega)$ and for any $s \geq 1$, the Poincar\'e--Sobolev inequality 
\[
\inf\limits_{c \in \mathbb R}\left(\int\limits_\Omega |f(y)-c|^sdy\right)^{\frac{1}{s}} \leq B_{s,2}(\Omega)
\left(\int\limits_\Omega |\nabla f(y)|^2 dy\right)^{\frac{1}{2}}
\] 
holds with the constant
$$
B_{s,2}(\Omega) \leq K^{\frac{1}{2}} B_{\frac{\beta s}{\beta-1},2}(\mathbb D) \|J_{\varphi}\mid L_{\beta}(\mathbb D)\|^{\frac{1}{s}}.
$$
\end{enumerate}
 }

Here
$B_{r,2}(\mathbb D) \leq \left(2^{-1} \pi\right)^{\frac{2-r}{2r}}\left(r+2\right)^{\frac{r+2}{2r}}$, $r=\beta s/(\beta -1)$
is the exact constant in the Poincar\'e-Sobolev inequality for unit disc $\mathbb D$
$$
\inf\limits_{c \in \mathbb R}\left(\int\limits_{\mathbb D} |g(x)-c|^rdx\right)^{\frac{1}{r}} \leq B_{r,2}(\mathbb D)
\left(\int\limits_{\mathbb D} |\nabla g(x)|^2 dx\right)^{\frac{1}{2}}.
$$ 

\vskip 0.3cm

The description of compactness of Sobolev embedding operators in the terms of capacity integrals was obtained in \cite{M}. In the present work we give sufficient conditions of compactness of Sobolev embedding operators in the terms of quasiconformal geometry of domains.

The suggested method is based on the theory of composition operators \cite{U93,VU02}
and its applications to the Sobolev type embedding theorems \cite{GG94,GU09}. 

The following diagram illustrates this idea:

\[\begin{array}{rcl}
W^{1}_{2}(\Omega) & \stackrel{\varphi^*}{\longrightarrow} & W^{1}_{2}(\mathbb D) \\[2mm]
\multicolumn{1}{c}{\downarrow} & & \multicolumn{1}{c}{\downarrow} \\[1mm]
L_s(\Omega) & \stackrel{(\varphi^{-1})^*}{\longleftarrow} & L_r(\mathbb D).
\end{array}\]

Here the operator $\varphi^*$ defined by the composition rule
$\varphi^*(f)=f \circ \varphi$ is a bounded composition operator on Sobolev  
spaces induced by a homeomorphism $\varphi$ of $\mathbb D$ and $\Omega$ and
the operator $(\varphi^{-1})^*$ defined by the composition rule
$(\varphi^{-1})^*(f)=f \circ \varphi^{-1}$ is a bounded composition operator on
Lebesgue spaces. This method allows to transfer Poincar\'e-Sobolev inequalities 
from regular domains (for example, from the unit disc $\mathbb D$) to $\Omega$.

In the recent works we study composition operators on Sobolev spaces defined on
planar domains in connection with the conformal mappings theory \cite{GU12}. This connection
leads to weighted Sobolev  embeddings \cite{GU13,GU14} with the universal conformal weights.
Another application of conformal composition operators was given in \cite{BGU15} where the 
spectral stability problem for conformal regular domains was considered.

\section{Composition operators and quasiconformal mappings}

In this section we recall basic facts about composition operators on Lebesgue and Sobolev spaces and also the quasiconformal mappings theory. Let $\Omega\subset\mathbb R^n$, $n\geq 2$, be a domain.
For any $1\leq p<\infty$ we consider the Lebesgue space $L_p(\Omega)$ of measurable functions $f: \Omega \to \mathbb{R}$ equipped with the following norm:
\[
\|f\mid L_p(\Omega)\|=\biggr(\int\limits _{\Omega}|f(x)|^{p}\, dx\biggr)^{\frac{1}{p}}<\infty.
\]  

The following theorem about composition operators on Lebesgue spaces is well known (see, for example \cite{VU02}):
\begin{theorem}\label{COL}
Let $\varphi :\Omega \to \widetilde{\Omega}$ be a weakly differentiable homeomorphism between two domains $\Omega$ and $\widetilde{\Omega}$.
Then the composition operator
\[
\varphi^{*}: L_r(\widetilde{\Omega}) \to L_s(\Omega),\,\,\,1 \leq s \leq r< \infty,
\] 
is bounded, if and only if $\varphi^{-1}$ possesses the Luzin $N$-property and
\[
\biggr(\int\limits _{\widetilde{\Omega}}|J(y,\varphi^{-1})|^{\frac{r}{r-s}}\, dy\biggr)^{\frac{r-s}{rs}}=K< \infty,\,\,\,1 \leq s<r< \infty,
\]
\[
\esssup\limits_{y \in \widetilde{\Omega}}|J(y,\varphi^{-1})|^{\frac{1}{s}}=K< \infty,\,\,\,1 \leq s=r< \infty.
\] 
The norm of the composition operator $\|\varphi^{*}\|=K$.
\end{theorem} 

We consider the Sobolev space $W^1_p(\Omega)$, $1\leq p<\infty$,
as a Banach space of locally integrable weakly differentiable functions
$f:\Omega\to\mathbb{R}$ equipped with the following norm: 
\[
\|f\mid W^1_p(\Omega)\|=\biggr(\int\limits _{\Omega}|f(x)|^{p}\, dx\biggr)^{\frac{1}{p}}+
\biggr(\int\limits _{\Omega}|\nabla f(x)|^{p}\, dx\biggr)^{\frac{1}{p}}.
\]
Recall that the Sobolev space $W^1_p(\Omega)$ coincides with the closer of the space of smooth functions $C^{\infty}(\Omega)$ in the norm of $W^1_p(\Omega)$.

We consider also the homogeneous seminormed Sobolev space $L^1_p(\Omega)$, $1\leq p<\infty$,
of locally integrable weakly differentiable functions $f:\Omega\to\mathbb{R}$ equipped
with the following seminorm: 
\[
\|f\mid L^1_p(\Omega)\|=\biggr(\int\limits _{\Omega}|\nabla f(x)|^{p}\, dx\biggr)^{\frac{1}{p}}.
\]

Recall the notion of the $p$-capacity of a set $E\subset \Omega$. Let $\Omega$ be a domain in $\mathbb R^n$ and a compact $F\subset\Omega$. The $p$-capacity of the compact $F$ is defined by
$$
\cp_p(F;\Omega)=\inf\{\|f|L^1_p(\Omega\|^p,\,\,f\geq 1\,\, \text{on}\,\, F, \,\,f\in C_0(\Omega)\}.
$$
By the similar way we can define $p$-capacity of open sets.

For arbitrary set $E\subset\Omega$ we define a inner $p$-capacity as 
$$
\underline{\cp_p}(E;\Omega)=\sup\{\cp_p(e;\Omega),\,\,e\subset E\subset\Omega,\,\, e\,\,\text{is a compact}\},
$$
and a outer $p$-capacity as 
$$
\overline{\cp_p}(E;\Omega)=\inf\{\cp_p(U;\Omega),\,\,E\subset U\subset\Omega,\,\, U\,\,\text{is an open set}\}.
$$
A set $E\subset\Omega$ is called $p$-capacity measurable, if $\underline{\cp_p}(E;\Omega)=\overline{\cp_p}(E;\Omega)$. The value
$$
\cp_p(E;\Omega)=\underline{\cp_p}(E;\Omega)=\overline{\cp_p}(E;\Omega)
$$
is called the $p$-capacity of the set $E\subset\Omega$.

By the standard definition functions of the class $L^1_p(\Omega)$ are defined only up to a set of measure zero, but they can be redefined quasi-everywhere i.~e. up to a set of conformal capacity zero. Indeed, every function $u\in L^1_p(\Omega)$ has a unique quasi-continuous representation $\tilde{u}\in L^1_p(\Omega)$. A function $\tilde{u}$ is termed quasi-continuous if for any $\varepsilon >0$ there is an open  set $U_{\varepsilon}$ such that the conformal capacity of $U_{\varepsilon}$ is less then $\varepsilon$ and the function  $\tilde{u}$ is continuous on the set $\Omega\setminus U_{\varepsilon}$ 
(see, for example \cite{HKM,M}). 

Let $\Omega$ and $\widetilde{\Omega}$ be domains in $\mathbb R^n$. We say that
a homeomorphism $\varphi:\Omega\to\widetilde{\Omega}$ induces by the composition rule $\varphi^{\ast}(f)=f\circ\varphi$ a bounded composition operator 
\[
\varphi^{\ast}:L^1_p(\widetilde{\Omega})\to L^1_q(\Omega),\,\,\,1\leq q\leq p\leq\infty,
\]
if the composition $\varphi^{\ast}(f)\in L^1_q(\Omega)$
is defined quasi-everywhere in $\Omega$ and there exists a constant $K_{p,q}(\Omega)<\infty$ such that 
\[
\|\varphi^{\ast}(f)\mid L^1_q(\Omega)\|\leq K_{p,q}(\Omega)\|f\mid L^1_p(\widetilde{\Omega})\|
\]
for
any function $f\in L^1_p(\widetilde{\Omega})$ \cite{VU04}.

Let $\Omega\subset\mathbb R^n$ be an open set. A mapping $\varphi:\Omega\to\mathbb R^n$ belongs to $L^1_{p,\loc}(\Omega)$, 
$1\leq p\leq\infty$, if its coordinate functions $\varphi_j$ belong to $L^1_{p,\loc}(\Omega)$, $j=1,\dots,n$.
In this case the formal Jacobi matrix
$D\varphi(x)=\left(\frac{\partial \varphi_i}{\partial x_j}(x)\right)$, $i,j=1,\dots,n$,
and its determinant (Jacobian) $J(x,\varphi)=\det D\varphi(x)$ are well defined at
almost all points $x\in \Omega$. The norm $|D\varphi(x)|$ of the matrix
$D\varphi(x)$ is the norm of the corresponding linear operator $D\varphi (x):\mathbb R^n \rightarrow \mathbb R^n$ defined by the matrix $D\varphi(x)$.

Let $\varphi:\Omega\to\widetilde{\Omega}$ be weakly differentiable in $\Omega$. The mapping $\varphi$ is the mapping of finite distortion if $|D\varphi(z)|=0$ for almost all $x\in Z=\{z\in\Omega : J(x,\varphi)=0\}$.

A mapping $\varphi:\Omega\to\mathbb R^n$ possesses the Luzin $N$-property if a image of any set of measure zero has measure zero.
Mote that any Lipschitz mapping possesses the Luzin $N$-property.

The following theorem gives the analytic description of composition operators on Sobolev spaces:

\begin{theorem}
\label{CompTh} \cite{U93, VU02} A homeomorphism $\varphi:\Omega\to\widetilde{\Omega}$
between two domains $\Omega$ and $\widetilde{\Omega}$ induces a bounded composition
operator 
\[
\varphi^{\ast}:L^1_p(\widetilde{\Omega})\to L^1_q(\Omega),\,\,\,1\leq p<\infty,
\]
 if and only if $\varphi\in W_{1,\loc}^{1}(\Omega)$, has finite distortion,
and 
$$
K_{p,q}(\Omega)=\left(\int\limits_\Omega \left(\frac{|D\varphi(x)|^p}{|J(x,\varphi)|}\right)^\frac{q}{p-q}~dx\right)^\frac{p-q}{pq}<\infty.
$$
\end{theorem}

Recall that a homeomorphism $\varphi: \Omega\to \widetilde{\Omega}$ is called a $K$-quasiconformal mapping if $\varphi\in W^1_{n,\loc}(\Omega)$ and there exists a constant $1\leq K<\infty$ such that
$$
|D\varphi(x)|^n\leq K |J(x,\varphi)|\,\,\text{for almost all}\,\,x\in\Omega.
$$

Quasiconformal mappings have a finite distortion, i.~e.  $D\varphi(x)=0$ for almost all points $x$
that belongs to set $Z=\{x\in \Omega:J(x,\varphi)=0\}$ because any quasiconformal mapping possesses Luzin $N$-property and an inverse mapping is also quasiconformal.

If $\varphi : \Omega \to \widetilde{\Omega}$ is a $K$-quasiconformal mapping then $\varphi$ is differentiable almost everywhere in $\Omega$ and
$$
|J(x,\varphi)|=J_{\varphi}(x):=\lim\limits_{r\to 0}\frac{|\varphi(B(x,r))|}{|B(x,r)|}\,\,\text{for almost all}\,\,x\in\Omega.
$$   

For any planar $K$-quasiconformal homeomorphism $\varphi : \Omega \to \widetilde{\Omega}$,
the following sharp results is known: $J(x,\varphi) \in L_{p,\loc}(\widetilde{\Omega})$ 
for any $p<K/(K-1)$ (\cite{Ast}).

If $K\equiv 1$ then $1$-quasiconformal homeomorphisms are conformal mappings and in the space $\mathbb R^n$, $n\geq 3$, are exhausted by M\"obius transformations.  

\begin{definition}
We call a bounded domain $\Omega\subset\mathbb R^2$ as $(r,q)$-Poincar\'e domain, $1 \leq q,r \leq \infty$, if the Poincar\'e--Sobolev inequality
$$
\inf_{c \in \mathbb{R}} ||g-c\,|\,L_r(\Omega)|| \leq B_{r,q}(\Omega) ||\nabla g\,|\,L_q(\Omega)||
$$
holds for any $g \in L_q^1(\Omega)$ with the Poincar\'e constant $B_{r,q}(\Omega)< \infty$. The unit disc $\mathbb D\subset\mathbb R^2$ is an example of the
$(r,2)$-embedding domain for all $r \geq 1$.  
\end{definition}

The following theorem gives a characterization of composition operators in the classical Sobolev spaces $W_p^1$ (see, for example \cite{GG94, GU09, GU16}): 
\begin{theorem}\label{COS}
Let $\Omega\subset\mathbb R^n$ be an $(r,q)$-Poincar\'e domain for some $1 \leq q \leq r \leq \infty$ and a domain $\widetilde{\Omega}$ has finite measure.
Suppose that a homeomorphism $\varphi:\Omega\to\widetilde{\Omega}$ induces a bounded composition operator
\[
\varphi^{*}: L_p^1(\widetilde{\Omega}) \to L_q^1(\Omega),\,\,\,1 \leq q \leq p< \infty,
\] 
and the inverse homeomorphism $\varphi^{-1}:\widetilde{\Omega}\to\Omega$ induces a bounded composition operator
\[
(\varphi^{-1})^{*}: L_r(\Omega) \to L_s(\widetilde{\Omega}),\,\,\,1 \leq s \leq r< \infty,
\]
for some $p \leq s \leq r$. 

Then $\varphi:\Omega\to\widetilde{\Omega}$ induces a bounded composition operator
\[
\varphi^{*}: W_p^1(\widetilde{\Omega}) \to W_q^1(\Omega),\,\,\,1 \leq q \leq p< \infty,
\] 
\end{theorem}

This theorem allows us to obtain compactness of the Sobolev embedding operator in quasiconformal regular domains.

\section{Poincar\'e-Sobolev inequalities}

\textbf{Weighted Poincar\'e-Sobolev inequalities}. Let $\Omega \subset \mathbb R^2$
be a planar domain and let $v : \Omega \to \mathbb R$ be a real valued
function, $v>0$ a.~e. in $\Omega$. We consider the weighted Lebesgue space $L_p(\Omega,v)$, $1\leq p<\infty$, 
of measurable functions $f: \Omega \to \mathbb R$  with the finite norm
$$
\|f\,|\,L_{p}(\Omega,v)\|:= \left(\int\limits_\Omega|f(x)|^pv(x)dx \right)^{\frac{1}{p}}< \infty.
$$
It is a Banach space for the norm $\|f\,|\,L_{p}(\Omega,v)\|$.

The following lemma gives connection between composition operators on Sobolev spaces and the quasiconformal mappings theory \cite{VG75}.

\begin{lemma} \label{L4.1}
A homeomorphism $\varphi: \Omega\to \widetilde{\Omega}$ is a $K$-quasiconformal mapping if and only if $\varphi$ generates by the composition rule $\varphi^{\ast}(f)=f\circ\varphi$ an isomorphism of Sobolev spaces $L^1_n(\Omega)$ and $L^1_n(\widetilde{\Omega})$:
$$
\|\varphi^{\ast}(f) \mid L^1_n(\Omega)\|\leq K^{\frac{1}{n}}\|f \mid L^1_n(\widetilde{\Omega})\|
$$
for any $f\in L^1_n(\widetilde{\Omega})$.
\end{lemma}

On the base of this lemma we prove the universal weighted Poincar\'e-Sobolev inequality which is correct for any simply connected planar domain with non-empty boundary.

\begin{theorem}\label{T4.2}
Suppose that $\Omega\subset\mathbb R^2$ is a simply connected domain with non-empty boundary and $h(y) =|J(y,\varphi^{-1})|$ is 
the quasiconformal weight defined by a $K$-quasiconformal mapping 
$\varphi : \mathbb D \to \Omega$.Then for every function $f \in W^{1}_{2}(\Omega)$,
the inequality
\[
\inf\limits_{c \in \mathbb R}\left(\int\limits_\Omega |f(y)-c|^rh(y)dy\right)^{\frac{1}{r}} \leq B_{r,2}(\Omega,h)
\left(\int\limits_\Omega |\nabla f(y)|^2 dy\right)^{\frac{1}{2}}
\]
holds for any $r \geq 1$ with the constant
$$
B_{r,2}(\Omega,h) \leq K^{\frac{1}{2}} \cdot B_{r,2}(\mathbb D) \leq \left(2^{-1} \pi\right)^{\frac{2-r}{2r}}\left(r+2\right)^{\frac{r+2}{2r}}K^{\frac{1}{2}}.
$$
\end{theorem}

Here $B_{r,2}(\mathbb D)$ is the best constant in the (non-weight) Poincar\'e-Sobolev inequality in the unit disc
$\mathbb D \subset \mathbb R^2$  with the upper estimate (see, for example, \cite{GT77,GU16}):
$$
B_{r,2}(\mathbb D) \leq \left(2^{-1} \pi\right)^{\frac{2-r}{2r}}\left(r+2\right)^{\frac{r+2}{2r}}. 
$$ 

\begin{proof} 
By \cite{Ahl66} there exists a $K$-quasiconformal homeomorphism $\varphi : \mathbb D \to \Omega$. 
Then by Lemma \ref{L4.1} the inequality  
\begin{equation}\label{IN2.1}
||\nabla (f \circ \varphi) \,|\, L_{2}(\mathbb D)|| \leq K^{\frac{1}{2}} ||\nabla f \,|\, L_{2}(\Omega)||
\end{equation}
holds for every function $f \in L^{1}_{2}(\Omega)$.

Let $f \in L^{1}_{2}(\Omega) \cap C^1(\Omega)$. Then the function $g=f \circ \varphi$ is defined almost everywhere 
in $\mathbb D$ and belongs to the Sobolev space $L^{1}_{2}(\mathbb D)$ \cite{VGR}. Hence, by the Sobolev embedding 
theorem $g=f \circ \varphi \in W^{1,2}(\mathbb D)$ \cite{M} and the classical Poincar\'e-Sobolev inequality,
\begin{equation}\label{IN2.3}
\inf_{c \in \mathbb R}||f \circ \varphi -c \,|\, L_{r}(\mathbb D)|| \leq B_{r,2}(\mathbb D) ||\nabla (f \circ \varphi) \,|\, L_{2}(\mathbb D)||
\end{equation}
holds for any $r \geq 1$. 

Denote by 
$h(y):=|J(y,\varphi^{-1})|$  quasiconformal weight in $\Omega$.
Using the change of variable formula for quasiconformal mappings \cite{VGR}, the classical Poincar\'e-Sobolev inequality for the unit disc 
$$
\inf\limits_{c \in \mathbb R}\left(\int\limits_{\mathbb D} |g(x)-c|^rdx\right)^{\frac{1}{r}} \leq B_{r,2}(\mathbb D)
\left(\int\limits_{\mathbb D} |\nabla g(x)|^2 dx\right)^{\frac{1}{2}}
$$
and inequality \eqref{IN2.1}, we get
\begin{multline*}
\inf\limits_{c \in \mathbb R}\left(\int\limits_\Omega |f(y)-c|^rh(y)dy\right)^{\frac{1}{r}} =
\inf\limits_{c \in \mathbb R}\left(\int\limits_\Omega |f(y)-c|^r |J(y,\varphi^{-1})| dy\right)^{\frac{1}{r}} \\
{} = \inf\limits_{c \in \mathbb R}\left(\int\limits_{\mathbb D} |g(x)-c|^rdx\right)^{\frac{1}{r}} \leq B_{r,2}(\mathbb D) 
\left(\int\limits_{\mathbb D} |\nabla g(x)|^2 dx\right)^{\frac{1}{2}} \\
{} \leq K^{\frac{1}{2}} B_{r,2}(\mathbb D) 
\left(\int\limits_{\Omega} |\nabla f(y)|^2 dy\right)^{\frac{1}{2}}.
\end{multline*}

Approximating an arbitrary function $f \in W^{1}_{2}(\Omega)$ by smooth functions we have
$$
\inf\limits_{c \in \mathbb R}\left(\int\limits_\Omega |f(y)-c|^rh(y)dy\right)^{\frac{1}{r}} \leq
B_{r,2}(\Omega,h) \left(\int\limits_{\Omega} |\nabla f(y)|^2 dy\right)^{\frac{1}{2}}
$$
with the constant
$$
B_{r,2}(\Omega,h) \leq K^{\frac{1}{2}} \cdot B_{r,2}(\mathbb D) \leq \left(2^{-1} \pi\right)^{\frac{2-r}{2r}}\left(r+2\right)^{\frac{r+2}{2r}}K^{\frac{1}{2}}.
$$
\end{proof}

The property of the quasiconformal $\beta$-regularity implies the integrability of a 
Jacobian of quasiconformal mappings and therefore for any quasiconformal $\beta$-regular domain
we have the embedding of weighted Lebesgue spaces $L_r(\Omega,h)$ into non-weighted Lebesgue
spaces $L_s(\Omega)$ for $s=\frac{\beta -1}{\beta}r$.

\begin{lemma} \label{L4.3}
Let $\Omega$ be a $K$-quasiconformal $\beta$-regular domain.Then for any function
$f \in L_r(\Omega,h)$, $\beta / (\beta - 1) \leq r < \infty$, the inequality
$$
||f\,|\,L_s(\Omega)|| \leq \left(\int\limits_\mathbb D \big|J(x,\varphi)\big|^{\beta}~dx \right)^{{\frac{1}{\beta}} \cdot \frac{1}{s}} ||f\,|\,L_r(\Omega,h)||
$$
holds for $s=\frac{\beta -1}{\beta}r$.
\end{lemma}

\begin{proof}
By the assumptions of the lemma these exists a $K$-quasiconformal mapping $\varphi : \mathbb D \to \Omega$ such that
$$
\int\limits_\mathbb D \big|J(x,\varphi)\big|^{\beta}~dx  < +\infty,
$$
Let $s=\frac{\beta -1}{\beta}r$. 
Then using the change of variable formula for quasiconformal mappings \cite{VGR}, H\"older's inequality with exponents $(r,rs/(r-s))$ and the equality 
$|J(y,\varphi^{-1})| = h(y)$, we obtain
\begin{multline*}
||f\,|\,L_s(\Omega)|| \\
{} = \left(\int\limits_{\Omega} |f(y)|^s dy\right)^{\frac{1}{s}} = 
\left(\int\limits_{\Omega} |f(y)|^s \big|J(y,\varphi^{-1})\big|^{\frac{s}{r}} \big|J(y,\varphi^{-1})\big|^{-\frac{s}{r}}~dy\right)^{\frac{1}{s}} \\
{} \leq \left(\int\limits_{\Omega} |f(y)|^r |J(y,\varphi^{-1})| dy\right)^{\frac{1}{r}}
 \left(\int\limits_{\Omega} \big|J(y,\varphi^{-1})\big|^{- \frac{s}{r-s}}~ dy\right)^{\frac{r-s}{rs}} \\
{} \leq \left(\int\limits_{\Omega} |f(y)|^r h(y)~ dy\right)^{\frac{1}{r}} 
\left(\int\limits_{\mathbb D} \big|J(x,\varphi)\big|^{\frac{r}{r-s}}~ dx\right)^{\frac{r-s}{rs}} \\
{} = \left(\int\limits_{\Omega} |f(y)|^r h(y)~ dy\right)^{\frac{1}{r}}
\left(\int\limits_\mathbb D \big|J(x,\varphi)\big|^{\beta}~dx \right)^{{\frac{1}{\beta}} \cdot \frac{1}{s}}.
\end{multline*}
 
\end{proof}

The following theorem is the main technical tool of this paper:

\vskip 0.3cm
\noindent
{\bf Theorem B.}
\textit{Let $\Omega\subset\mathbb R^2$ be a $K$-quasiconformal $\beta$-regular domain. Then:
\begin{enumerate}[(1)]
\item the embedding operator
\[
i_{\Omega}:W_2^1(\Omega) \hookrightarrow L_s(\Omega)
\]
is compact for any $s \geq 1$;
\item for any function $f \in W^{1}_{2}(\Omega)$ and for any $s \geq 1$, the Poincar\'e--Sobolev inequality 
\[
\inf\limits_{c \in \mathbb R}\left(\int\limits_\Omega |f(y)-c|^sdy\right)^{\frac{1}{s}} \leq B_{s,2}(\Omega)
\left(\int\limits_\Omega |\nabla f(y)|^2 dy\right)^{\frac{1}{2}}
\] 
holds with the constant
$$
B_{s,2}(\Omega) \leq K^{\frac{1}{2}} B_{\frac{\beta s}{\beta-1},2}(\mathbb D) \|J_{\varphi}\mid L_{\beta}(\mathbb D)\|^{\frac{1}{s}}.
$$
\end{enumerate}
 }

\begin{proof}
Let $s \geq 1$. Since $\Omega$ is a $K$-quasiconformal $\beta$-regular domain then there exists a  $K$-quasiconformal mapping 
$\varphi : \mathbb D \to \Omega$ such that
$$
\int\limits_\mathbb D |J(x, \varphi)|^{\beta}~dx < \infty \quad\text{for some}\quad \beta >1.
$$

By Theorem~\ref{COL} the composition operator
\[
(\varphi^{-1})^{*}:L_r(\mathbb D)\to L_s(\Omega)
\]
is bounded if
\[
\biggr(\int\limits _{\mathbb D}|J(x,\varphi)|^{\frac{r}{r-s}}\, dx\biggr)^{\frac{r-s}{rs}}< \infty.
\]
Because $\Omega$ is a $K$-quasiconformal $\beta$-regular domain this condition holds for $r/(r-s)=\beta$ i.e. for $r=\beta s /(\beta -1)$.

Since the mapping  $\varphi: \mathbb D \to \Omega$ induced a bounded composition operator
\[
\varphi^{*}:L^1_2(\Omega)\to L^1_2(\mathbb D),
\]
then by Theorem~\ref{COS} the composition operator
\[
\varphi^{*}:W^1_2(\Omega)\to W^1_2(\mathbb D),
\]
is bounded.

For the unit disc $\mathbb D$ the embedding operator
\[
i_{\mathbb D}:W_2^1(\mathbb D) \hookrightarrow L_r(\mathbb D),
\]
is compact (see, for example \cite{M}) for any $r \geq 1$.

Therefore the embedding operator
\[
i_{\Omega}:W^1_2(\Omega)\to L_s(\Omega)
\]
is compact as a composition of bounded composition operators $\varphi^{*}$, $(\varphi^{-1})^{*}$ and the compact embedding operator $i_{\mathbb D}$.

Let $f\in W^1_2(\Omega)$. Then by Theorem \ref{T4.2} and Lemma \ref{L4.3} we obtain
\begin{multline*}
\inf_{c \in \mathbb R} \left(\int\limits_{\Omega} |f(y)-c|^s dy\right)^{\frac{1}{s}} \\
{} \leq \left(\int\limits_\mathbb D \big|J(x,\varphi)\big|^{\beta}~dx \right)^{{\frac{1}{\beta}} \cdot \frac{1}{s}}
\inf_{c \in \mathbb R} \left(\int\limits_{\Omega} |f(y)-c|^r h(y) dy\right)^{\frac{1}{r}} \\
{} \leq K^{\frac{1}{2}} B_{r,2}(\mathbb D)
\left(\int\limits_\mathbb D \big|J(x,\varphi)\big|^{\beta}~dc \right)^{{\frac{1}{\beta}} \cdot \frac{1}{s}}
\left(\int\limits_\Omega |\nabla f(y)|^2 dy\right)^{\frac{1}{2}}
\end{multline*} 
for $s\geq 1$.
\end{proof}

The following theorem gives compactness of the embedding operator in the case $\beta = \infty$:

\begin{theorem}\label{T4.5}
Let $\Omega$ is a $K$-quasiconformal $\infty$-regular domain. Then:
\begin{enumerate}[(1)]
\item The embedding operator
\[
i_{\Omega}:W_2^1(\Omega) \hookrightarrow L_2(\Omega),
\]
is compact.

\item For any function $f \in W^{1}_{2}(\Omega)$, the Poincar\'e--Sobolev inequality 
\[
\inf\limits_{c \in \mathbb R}\left(\int\limits_\Omega |f(y)-c|^2dy\right)^{\frac{1}{2}} \leq B_{2,2}(\Omega)
\left(\int\limits_\Omega |\nabla f(y)|^2 dy\right)^{\frac{1}{2}}
\]
holds.

\item The following estimate is correct: $B_{2,2}(\Omega) \leq K^{\frac{1}{2}} B_{2,2}(\mathbb D) \big\|J_{\varphi}\mid L_{\infty}(\mathbb D)\big\|^{\frac{1}{2}}$.
Here $B_{2,2}^2(\mathbb D)=1/\mu_1(\mathbb D)$ is the exact for the Poincar\'e inequality in the unit disc.
\end{enumerate}
\end{theorem}

\begin{proof}
Since $\Omega$ is a $K$-quasiconformal $\infty$-regular domain then there exists a $K$-quasiconformal mapping 
$\varphi:\mathbb D \to \Omega$ such that
\[
\big\|J_{\varphi}\mid L_{\infty}(\mathbb D)\big\|=\esssup\limits_{x \in \mathbb D}|J(x,\varphi)|<\infty.
\]

Hence by Theorem~\ref{COL} the composition operator
\[
(\varphi^{-1})^{*}:L_2(\mathbb D)\to L_2(\Omega)
\]
is bounded.

Since the mapping  $\varphi: \mathbb D \to \Omega$ induced a bounded composition operator
\[
\varphi^{*}:L^1_2(\Omega)\to L^1_2(\mathbb D),
\]
then by Theorem~\ref{COS} the composition operator
\[
\varphi^{*}:W^1_2(\Omega)\to W^1_2(\mathbb D),
\]
is bounded.

For the unit disc $\mathbb D$, the embedding operator
\[
i_{\mathbb D}:W_2^1(\mathbb D) \hookrightarrow L_2(\mathbb D),
\]
is compact (see, for example \cite{M}).

Therefore the embedding operator
\[
i_{\Omega}:W^1_2(\Omega)\to L_2(\Omega),
\]
is compact as a composition of bounded composition operators $\varphi^{*}$, $(\varphi^{-1})^{*}$ and the compact embedding operator $i_{\mathbb D}$:
\[
i_{\mathbb D}:W_2^1(\mathbb D) \hookrightarrow L_2(\mathbb D).
\]

The first part of this theorem is proved.

For every function $f \in W^1_2(\Omega) \cap C^1(\Omega)$ and $g=f \circ \varphi \in W^1_2(\mathbb D)$,
the following inequality are correct:

\begin{multline*}
\inf\limits_{c \in \mathbb R}\left(\int\limits_\Omega |f(y)-c|^2dy\right)^{\frac{1}{2}} =
\inf\limits_{c \in \mathbb R}\left(\int\limits_\Omega |f(y)-c|^2 |J(y,\varphi^{-1})|^{-1} |J(y,\varphi^{-1})|~dy\right)^{\frac{1}{2}} \\
{} \leq \big\|J_{\varphi^{-1}}\mid L_{\infty}(\Omega)\big\|^{-\frac{1}{2}}
\inf\limits_{c \in \mathbb R}\left(\int\limits_\Omega |f(y)-c|^2 |J(y,\varphi^{-1})|~dy\right)^{\frac{1}{2}}.
\end{multline*}

Because quasiconformal mappings possess the Luzin $N$-property, then
$$
\frac{1}{|J(y,\varphi^{-1})|}=|J(x,\varphi)| \,\,\text{for almost all}\,\,x\in \mathbb D\,\,\text{and for almost all}\,\, y\in\Omega.
$$ 

Hence

\begin{multline*}
\inf\limits_{c \in \mathbb R}\left(\int\limits_\Omega |f(y)-c|^2dy\right)^{\frac{1}{2}} \\
{} \leq \big\|J_{\varphi}\mid L_{\infty}(\mathbb D)\big\|^{\frac{1}{2}}
\inf\limits_{c \in \mathbb R}\left(\int\limits_\Omega |f(y)-c|^2 |J(y,\varphi^{-1})|~dy\right)^{\frac{1}{2}}.
\end{multline*}

Using the change of variable formula for quasiconformal mappings \cite{VGR}, the Poincar\'e inequality in the unit disc and the inequality~\eqref{IN2.1} finally we obtain
\begin{multline*}
\inf\limits_{c \in \mathbb R}\left(\int\limits_\Omega |f(y)-c|^2dy\right)^{\frac{1}{2}} \\
{} \leq \big\|J_{\varphi}\mid L_{\infty}(\mathbb D)\big\|^{\frac{1}{2}} 
\inf\limits_{c \in \mathbb R}\left(\int\limits_\Omega |f(y)-c|^2 |J(y,\varphi^{-1})|~dy\right)^{\frac{1}{2}} \\
{} = \big\|J_{\varphi}\mid L_{\infty}(\mathbb D)\big\|^{\frac{1}{2}} 
\inf\limits_{c \in \mathbb R}\left(\int\limits_\mathbb D |g(x)-c|^2~dx\right)^{\frac{1}{2}} \\
{} \leq \big\|J_{\varphi}\mid L_{\infty}(\mathbb D)\big\|^{\frac{1}{2}} 
B_{2,2}(\mathbb D)\left(\int\limits_\mathbb D |\nabla g(x)|^2~dx\right)^{\frac{1}{2}} \\
{} \leq K^{\frac{1}{2}} B_{2,2}(\mathbb D) \big\|J_{\varphi}\mid L_{\infty}(\mathbb D)\big\|^{\frac{1}{2}}
\left(\int\limits_\Omega |\nabla f(y)|^2~dy\right)^{\frac{1}{2}}. 
\end{multline*}
\end{proof}

\section{Eigenvalue Problem for Neumann-Laplacian}

The eigenvalue problem for the free vibrating membrane is equivalent to the corresponding spectral problem
for the Neumann--Laplace operator. The classical formulation of the spectral problem
for the Neumann--Laplace operator in smooth domains in the following:
\[
\begin{cases}
-\Delta u=\mu u & \text{in $\Omega$}\\
\frac{\partial u}{\partial n}=0 & \text{on $\partial \Omega$}.
\end{cases} 
\]

Because quasiconformal regular domain are not necessary smooth, the weak statement of the
spectral problem for the Neumann-Laplace operator is convenient: a function $u$ solves
the previous problem iff $u \in W_2^1(\Omega)$ and 
$$
\int\limits _{\Omega} \nabla u(x) \cdot \nabla v(x)\,dx =
\mu \int\limits _{\Omega} u(x) v(x)\,dx
$$ 
for all $v \in W^{1}_{2}(\Omega)$.

By the Min--Max Principle \cite{D95}, the inverse to the first eigenvalue is equal to the square of the exact
constant in the Poincar\'e inequality:
$$
\inf_{c \in \mathbb{R}} ||f-c\,|\,L_2(\Omega)|| \leq B_{2,2}(\Omega) ||\nabla f\,|\,L_2(\Omega)||,  \quad f \in W^{1}_{2}(\Omega). 
$$

\vskip 0.3cm
\noindent
{\bf Theorem A.}
\textit{Let $\Omega \subset \mathbb R^2$ be a $K$-quasiconformal $\beta$-regular domain. Then the spectrum of the Neumann--Laplace operator in $\Omega$ is discrete,
and can be written in the form of a non-decreasing sequence:
\[
0=\mu_0(\Omega)<\mu_1(\Omega)\leq \mu_2(\Omega)\leq \ldots \leq \mu_n(\Omega)\leq \ldots ,
\]
and
\begin{multline*}
\frac{1}{\mu_1(\Omega)} \leq K B_{\frac{2\beta}{\beta -1},2}(\mathbb D)
\left(\int\limits_\mathbb D \big|J(x,\varphi)\big|^{\beta}~dx \right)^{{\frac{1}{\beta}}} \\
{} \leq
\frac{4K}{\sqrt[\beta]{\pi}} \left(\frac{2\beta -1}{\beta -1}\right)^{\frac{2 \beta-1}{\beta}} \big\|J_{\varphi}\mid L_{\beta}(\mathbb D)\big\|,
\end{multline*}
where $\varphi:\mathbb D \to \Omega$ is the $K$-quasiconformal mapping. }

\begin{proof}
By Theorem~B in the case $s=2$, the embedding operator
$$
i_{\Omega}:W^1_2(\Omega)\to L_2(\Omega)
$$
is compact.

Therefore the spectrum of the Neumann--Laplace operator is discrete and can be written in the form of a non-decreasing sequence.

By the same theorem and the Min-Max principle we have
\[
\inf_{c \in \mathbb R} \left(\int\limits_{\Omega} |f(y)-c|^2 dy\right) \leq B^2_{2,2}(\Omega)
\int\limits_\Omega |\nabla f(y)|^2 dy,
\]
where
\[
B_{2,2}(\Omega) \leq K^{\frac{1}{2}} B_{r,2}(\mathbb D)
\left(\int\limits_\mathbb D \big|J(x,\varphi)\big|^{\beta}~dx \right)^{{\frac{1}{2\beta}}}. 
\]

Hence
\[
\frac{1}{\mu_1(\Omega)} \leq K B^2_{r,2}(\mathbb D)
\left(\int\limits_\mathbb D \big|J(x,\varphi)\big|^{\beta}~dy \right)^{{\frac{1}{\beta}}}.
\]
By the upper estimate of the Poincar\'e constant in the unit disc (see, for example, \cite{GT77,GU16})
\[
B_{r,2}(\mathbb D) \leq \left(2^{-1} \pi\right)^{\frac{2-r}{2r}}\left(r+2\right)^{\frac{r+2}{2r}}.
\]

Recall that by Theorem~B, $r=2\beta /(\beta -1)$. In this case
\[
B_{\frac{2\beta}{\beta -1},2}(\mathbb D) \leq 2\pi^{-\frac{1}{2\beta}} \left(\frac{2\beta -1}{\beta -1}\right)^{\frac{2\beta -1}{2\beta}}.
\]
Thus
\begin{multline*}
\frac{1}{\mu_1(\Omega)} \leq K B_{\frac{2\beta}{\beta -1},2}(\mathbb D)
\left(\int\limits_\mathbb D \big|J(x,\varphi)\big|^{\beta}~dx \right)^{{\frac{1}{\beta}}}\\
{} \leq
\frac{4K}{\sqrt[\beta]{\pi}} \left(\frac{2\beta -1}{\beta -1}\right)^{\frac{2 \beta-1}{\beta}} \big\|J_{\varphi}\mid L_{\beta}(\mathbb D)\big\|.
\end{multline*}
\end{proof}

In case $K$-quasiconformal $\infty$-regular domains we have: 

\begin{theorem}\label{T4.7}
Let $\Omega \subset \mathbb R^2$ be a $K$-quasiconformal $\beta$-regular domain for $\beta = \infty$. Then the spectrum of the Neumann--Laplace operator in $\Omega$ is discrete,
and can be written in the form of a non-decreasing sequence:
\[
0=\mu_0(\Omega)<\mu_1(\Omega)\leq \mu_2(\Omega)\leq \ldots \leq \mu_n(\Omega)\leq \ldots ,
\]
and
\begin{equation}
\frac{1}{\mu_1(\Omega)} \leq K B^2_{2,2}(\mathbb D) \big\|J_{\varphi}\mid L_{\infty}(\mathbb D)\big\|
= \frac{K}{j^2_{1,1}}\big\|J_{\varphi}\mid L_{\infty}(\mathbb D)\big\|,
\end{equation}
where $j'_{1,1}$ is the first positive zero the derivative of the Bessel function $J_1$, and 
$\varphi:\mathbb D \to \Omega$ is the $K$-quasiconformal mapping.
\end{theorem}

As an application of Theorem~\ref{T4.7}, we obtain the lower estimates of the first non-trivial eigenvalue on the 
Neumann eigenvalue problem for the Laplace operator in a non-convex domains with a non-smooth boundaries.

\vskip 0.3cm
\noindent
$\mathbf{Example \, 1.}$
The homeomorphism 
$$
w= \left(|z|^{k-1}z+1\right)^2, \quad z=x+iy, \quad k\geq 1,
$$
is $k$-quasiconformal and maps the unit disc $\mathbb D$ onto the interior of the cardioid
$$
\Omega_c= \left\{(x,y) \in \mathbb R^2: (x^2+y^2-2x)^2-4(x^2+y^2)=0\right\}.
$$ 

We calculate the Jacobian of mapping $w$ by the formula 
$$
J(z,w)=|w_z|^2-|w_{\overline{z}}|^2.
$$
Here
$$
w_z=\frac{1}{2}\left(\frac{\partial w}{\partial x}-i\frac{\partial w}{\partial y}\right) \quad \text{and} \quad 
w_{\overline{z}}=\frac{1}{2}\left(\frac{\partial w}{\partial x}+i\frac{\partial w}{\partial y}\right).
$$

A straightforward calculation yields 
$$
w_z=(k+1)|z|^{k-1}\left(|z|^{k-1}z+1\right) \quad \text{and} \quad w_{\overline{z}}=(k-1)|z|^{k-3}z^2\left(|z|^{k-1}z+1\right).
$$ 
Hence
$$
J(z,w)=4k|z|^{2k-2}\left(|z|^{2k}+|z|^{k-1}(z+\overline{z})+1\right).
$$
Then by Theorem~\ref{T4.7} we have
\[
\frac{1}{\mu_1(\Omega_c)} \leq  \frac{K}{j'^2_{1,1}} \esssup\limits_{|z|\leq 1}\big|J(z,w)\big| \leq \frac{16k^2}{j'^2_{1,1}}.
\]

\vskip 0.3cm
\noindent
$\mathbf{Example \, 2.}$
The homeomorphism 
$$
w= |z|^{k}z, \quad z=x+iy, \quad k\geq 0,
$$
is $(k+1)$-quasiconformal and maps the square 
$$
Q:=\left\{(x,y) \in \mathbb R^2:-\frac{\sqrt{2}}{2} <x< \frac{\sqrt{2}}{2},\, -\frac{\sqrt{2}}{2} <y< \frac{\sqrt{2}}{2}\right\}
$$ 
onto star-shaped domains $\Omega_{\varepsilon}^*$ with vertices $(\pm \sqrt{2}/2,\, \pm \sqrt{2}/2),
(\pm \varepsilon,\,0)$ and \\
$(0,\, \pm \varepsilon)$, where $\varepsilon = (\sqrt{2}/2)^{k+1}$.

\begin{figure}[h!]
\centering
\includegraphics[width=0.6\textwidth]{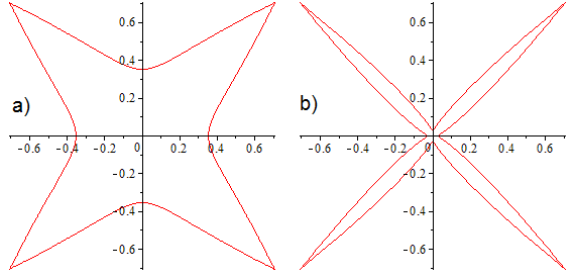}
\caption{Domains $\Omega_{\varepsilon}^*$ under $\varepsilon=\frac{1}{2\sqrt{2}}$ and  $\varepsilon=\frac{1}{32}$.}
\end{figure}

We calculate the partial derivatives of mapping $w$
$$
w_z=(\frac{k}{2}+1)|z|^{k} \quad \text{and} \quad w_{\overline{z}}=\frac{k}{2}|z|^{k-2}z^2.
$$ 
Thus
$$
J(z,w)=(k+1)|z|^{2k}.
$$

Because the square $Q$ is the quasiconformal $\infty$-regular domain, by Theorem~\ref{T4.7} we have
\begin{equation}\label{En1}
\frac{1}{\mu_2(\Omega_{\varepsilon}^*)} \leq 
B^2_{2,2}(Q) \cdot K \cdot \esssup\limits_{|z|\leq 1}|J(z,w)| \leq \frac{2(k+1)^2}{\pi^2}.
\end{equation}
Here $B_{2,2}(Q)=\sqrt{2}/\pi$ (see, for example, \cite{KN}) is the exact constant for the Poincar\'e inequality in the square Q.

In \cite{VV} (see, example 4.2) obtained estimates of constants in weighted Poincar\'e inequality for stars $\Omega_\epsilon^T$:
\begin{equation}\label{En2}
C_{p,\lambda}(\Omega_\epsilon^T)\leq 14 \bar{C}_p\left[\frac{8}{7}\left(\frac{1}{2}+\frac{p}{2\bar{C}_p}\right)\right]^\frac{1}{p},
\end{equation} 
where $p \geq 1$ and
$$
\bar{C}_1=\frac{1}{2}, \quad \bar{C}_2=\frac{1}{\pi}, \quad \bar{C}_p \leq 2\left(\frac{p}{2}\right)^\frac{1}{p}.
$$

Note that estimate~\eqref{En1} under $0 \leq k<6$ is a better by comparison with 
estimate~\eqref{En2} for $p=2$ $(C_{2,\lambda}(\Omega_\epsilon^T) \leq 4 \sqrt{7(1+2 \pi)}/\pi \approx 9,09117)$.

\vskip 0.3cm

{\bf Acknowledgments}:

The first author was supported by the United States-Israel Binational Science Foundation (BSF Grant No. 2014055).

\vskip 0.3cm

Department of Mathematics, Ben-Gurion University of the Negev, P.O.Box 653, Beer Sheva, 8410501, Israel 
 
\emph{E-mail address:} \email{vladimir@math.bgu.ac.il} \\           
       
 Department of Higher Mathematics and Mathematical Physics, Tomsk Polytechnic University, 634050 Tomsk, Lenin Ave. 30, Russia;
 Department of General Mathematics, Tomsk State University, 634050 Tomsk, Lenin Ave. 36, Russia

 \emph{Current address:} Department of Mathematics, Ben-Gurion University of the Negev, P.O.Box 653, 
  Beer Sheva, 8410501, Israel  
							
 \emph{E-mail address:} \email{vpchelintsev@vtomske.ru}   \\
			  
	Department of Mathematics, Ben-Gurion University of the Negev, P.O.Box 653, Beer Sheva, 8410501, Israel 
							
	\emph{E-mail address:} \email{ukhlov@math.bgu.ac.il


\vskip 0.3cm

\begin{thebibliography}{99}

\bibitem{Ahl66} 
Ahlfors,~L., {\em Lectures on quasiconformal mappings.} 
D. Van Nostrand Co., Inc., Toronto, Ont.-New York-London, 1966.

\bibitem{Ast} 
Astala,~K., Area distortion of quasiconformal mappings. 
{\it Acta Math.} 173 (1994), 37--60.

\bibitem{AK}  
Astala,~K., Koskela,~P., Quasiconformal mappings and global integrability of the derivative,
{\it J. Anal. Math.} 57 (1991), 203--220.

\bibitem{BM07} 
Beardon,~A.~F., Minda,~D., The hyperbolic metric and geometric function theory. 
Quasiconformal mappings and their applications, Narosa, New Delhi, (2007) 9--56.

\bibitem{Ber} 
Bertilsson,~D., {\em On Brennan's Conjecture
in conformal mapping.} Doctoral Thesis, Royal Institute of Technology,
Stockholm, Sweden (1999).

\bibitem{BCT15} 
Brandolini,~B., Chiacchio,~F., Trombetti,~C., Optimal lower lounds for eigenvalue of linear and nonlinear
Neumann problems. {\it Proc. of the Royal Soc. of Edinburgh}, 145A (2015), 31--45.

\bibitem{BIw} 
Bojarski,~B., Iwaniec,~T., Analytical foundation of the theoty of quasiconformal mappings
in $\mathbb R^n$. {\it Ann. Acad. Sci. Fenn. Ser. A. I. Math.} 8 (1983), 257--324.

\bibitem{BGU15} 
Burenkov,~V.~I., Gol'dshtein,~V., Ukhlov,~A., Conformal spectral stability for the Dirichlet--Laplacian.
 {\it Math. Nachr.} 288 (2015), 1822--1833.

\bibitem{BGU16} 
Burenkov,~V.~I., Gol'dshtein,~V., Ukhlov,~A., Conformal spectral stability for the Neumann--Laplacian.
 {\it Math. Nachr.} 289 (2016), 2133--2146.

\bibitem{D95} 
Davies,~E.~B., {\em Spectral theory and differential operators.} Cambridge University Press: Cambridge, 1995.

\bibitem{ENT} 
Esposito,~L., Nitsch,~C., Trombetti,~C., Best constants in Poincar\'e inequalities for convex domains.
{\it J. Convex Anal.} 20 (2013), 253--264.

\bibitem{FNT} 
Ferone,~V., Nitsch~C. and Trombetti,~C., 
A remark on optimal weighted Poincar\'e inequalities for convex domains. 
{\it Atti Accad. Naz. Lincei Rend. Lincei Mat. Appl.} 23 (2012), 467--475. 

\bibitem{G60} 
Gehring,~F.~W., The definitions and exceptional sets for quasiconformal mappings.
{\it Ann. Acad. Sci. Fenn. Math.} 281 (1960), 1--28.

\bibitem{GT77}
Gilbarg,~D., Trudinger,~N.S., {\em Elliptic Partial Differential Equations of Second Order.}
Springer-Verlag: Berlin-Heidelberg-New York, 1977.

\bibitem{GG94} 
Gol'dshtein,~V., Gurov,~L., Applications of change of variables operators for exact embedding theorems. 
{\it Integral Equ. Oper. Theory} 19 (1994), 1--24.

\bibitem{GU09} 
Gol'dshtein,~V., Ukhlov,~A., Weighted Sobolev spaces and embedding theorems. 
{\t Trans. Am. Math. Soc.} 361 (2009), 3829--3850. 

\bibitem{GU12} 
Gol'dshtein,~V., Ukhlov,~A., Brennan's conjecture for composition operators on Sobolev spaces. 
{\it Eurasian Math. J.} 3 (2012)(4), 35--43.

\bibitem{GU13} 
Gol'dshtein,~V., Ukhlov,~A., Conformal weights and Sobolev embeddings. 
{\it J. Math. Sci. (N.Y.)} 193 (2013), 202--210.

\bibitem{GU14} 
Gol'dshtein,~V., Ukhlov,~A., Brennan's Conjecture and universal Sobolev inequalities. 
{\it Bull. Sci. Math.} 138 (2014), 253--269.

\bibitem{GU16} 
Gol'dshtein,~V., Ukhlov,~A., On the first Eigenvalues of Free Vibrating Membranes in Conformal Regular Domains. 
{\it Arch. Rational Mech. Anal.} 221 (2016)(2), 893--915.

\bibitem{GU2016} 
Gol'dshtein,~V., Ukhlov,~A., Spectral estimates of the $p$-Laplace Neumann operator in conformal regular domains. 
{\it Transactions of A. Razmadze Math. Inst.} 170 (2016)(1), 137--148.

\bibitem{HShim} 
Hedelman,~H., Shimorin,~S., Weighted Bergman Spaces and the integral 
spectrum of conformal mappings. {\it Duke Mathematical J.} 127 (2005), 341-393.

\bibitem{HK95} 
Heinonen,~J., Koskela,~P., Weighted Sobolev and Poincar\'e inequalities and quasiregular mappings of polynomial type.
{\it Math. Scand.} 77 (1995), 251-271.

\bibitem{HKM}
Heinonen,~J., Kilpel\"ainen,~T., Martio,~O.,
{\em Nonlinear Potential Theory of Degenerate Elliptic Equations.}
Oxford Math. Monographs, Oxford Univ. Press, 1993.

\bibitem{KOT} 
Koskela,~P., Onninen,~J., Tyson,~J.~T., 
Quasihyperbolic boundary conditions and capacity: Poincar\'e domains. 
{\it Math. Ann.} 323 (2002), 811--830.

\bibitem{KN} 
Kuznetsov,~N.G., Nazarov,~A.I., Sharp constants in Poincar\'e, Steklov and related inequalities (a survey). 
{\it Mathematika} 61 (2015), 328--344. 

\bibitem{M} 
Maz'ya,~V., {\em Sobolev spaces: with applications to elliptic
partial differential equations.} Springer: Berlin/Heidelberg, 2010.

\bibitem{P60}
P\'olya,~G., 
On the eigenvalues of vibrating membranes. 
{\it Proc. London Math. Soc.} 11 (1961), 419--433.

\bibitem{PW} 
Payne,~L.~E., Weinberger,~H.~F., 
An optimal Poincar\'e inequality for convex domains. 
{\it Arch. Rat. Mech. Anal.} 5 (1960), 286--292.

\bibitem{U93} 
Ukhlov,~A., On mappings, which induce embeddings of
Sobolev spaces. {\it Siberian Math. J.} 34 (1993), 185--192.

\bibitem{V12} 
Valtorta,~D., Sharp estimate on the first eigenvalue of the p-Laplacian. 
{\it Nonlin. Analysis} 75 (2012), 4974--4994.

\bibitem{VV} 
Veeser,~A., Verf\"urth,~R., Poincar\'e constants for finite element stars. 
{\it IMA J. Numer. Anal.} 32 (2012)(1), 30--47. 

\bibitem{VG75}
Vodop'yanov,~S.~K., Gol'dstein,~V.~M., Lattice isomorphisms of the spaces $W^1_n$ and quasiconformal mappings
{\it Siberian Math. J.} 16 (1975), 224--246.

\bibitem{VGR}
Vodop'yanov,~S.~K., Gol'dstein,~V.~M., Reshetnyak,~Yu.~G.,
On geometric properties of functions with generalized first derivatives.
{\it Uspekhi Mat. Nauk} 34 (1979), 17--65. 

\bibitem{VU98} 
Vodop'yanov,~S.~K., Ukhlov,~A.~D., Sobolev spaces and $(P,Q)$-quasiconformal mappings of Carnot groups.
{\it Siberian Math. J.} 39 (1998), 665--682.

\bibitem{VU02} 
Vodop'yanov,~S.~K., Ukhlov,~A.~D., Superposition
operators in Sobolev spaces (in Russian).
{\it Izvestiya VUZ} 46 (2002)(4), 11--33.

\bibitem{VU04} 
Vodop'yanov,~S.~K., Ukhlov,~A.~D., Set functions and their applications in the theory of Lebesgue and Sobolev spaces. 
{\it Siberian Adv. in Math.} 14 (2004), 78--125.

 
\end{thebibliography}
\end{document}